\documentclass[a4paper,10pt]{article}

\usepackage{amsmath, amssymb, amsthm}
\usepackage{graphics}

\usepackage{psfrag, booktabs}
\usepackage{color, epsfig, float}
\usepackage[active]{srcltx}
\usepackage[bookmarksopen, colorlinks, linkcolor = blue, urlcolor = red,
            citecolor = red, menucolor = blue]{hyperref}

\newtheorem{theorem}{Theorem}
\newtheorem{lemma}{Lemma}[section]
\newtheorem{remark}[lemma]{Remark}

\newcommand\om{\omega}
\newcommand\la{\lambda}

\newcommand\F{\mathrm{F}}
\newcommand\nur{\mathcal{N}}

\newcommand\R{\mathbb{R}}

\newcommand\N{\mathbb{N}}
\newcommand\tm{\subseteq}
\newcommand{\nm}{\phantom{-}}

\newcommand\norm[1]{\|#1\|}
\newcommand\set[1]{\{#1\}}

\newcommand\f {\mathbf}

\title%[Kaczmarz methods for ill-posed equations I]
{Regularization of systems of nonlinear ill-posed equations: I.
Convergence Analysis}

\author%[M. Haltmeier,   A. Leit\~ao,  O. Scherzer]
{}

\begin{document}
%\centerline{\textbf{Propsed running head}}
%\centerline{Kaczmarz methods for regularizing  ill-posed equations I }

%\medskip
%\medskip
%\centerline{\textbf{Send proofs to}}

%\centerline{Richard Kowar}
%\centerline{Department of Computer Science, University of Innsbruck}
%\centerline{Technikerstr. 21a, A-6020 Innsbruck, Austria}
%\centerline{{\it E-mail address: } richard.kowar@uibk.ac.at}

%\newpage

\maketitle

%% Enter the first author's name and address:
\centerline{\scshape Markus Haltmeier}
{\footnotesize
\centerline{Department of Computer Science, University of Innsbruck}
   \centerline{Technikerstrasse 21a, A-6020 Innsbruck, Austria}
}
\medskip
\centerline{\scshape Antonio Leit\~ao}
{\footnotesize
\centerline{Department of Mathematics, Federal University of St. Catarina}
   \centerline{P.O. Box 476, 88040-900 Florian\'opolis, Brazil}
}
\medskip
\centerline{\scshape Otmar Scherzer}
{\footnotesize
\centerline{Department of Computer Science, University of Innsbruck}
\centerline{Technikerstrasse 21a, A-6020 Innsbruck, Austria}
\centerline{ and }
\centerline{RICAM, Austrian Academy of Sciences}
\centerline{ Altenbergerstrasse 69, A-4040 Linz, Austria.}
}
\medskip
%% The name of the associate editor will be entered by a editorial staff
%\centerline{(Communicated by XXX)}
%\medskip

\noindent
\textbf{2000 Mathematics Subject Classification.}
Primary: 65J20, 65J15; Secondary: 47J06. \\
\textbf{Keywords.}
Ill-posed systems; Landweber--Kaczmarz; Regularization.

%\newpage

\begin{abstract}
In this article we develop and analyze novel iterative regularization techniques for
the solution of systems of nonlinear ill--posed operator equations.  The basic idea consists
in considering separately each equation of this system and  incorporating a loping strategy.
The first technique is a Kaczmarz--type method, equipped with a
novel stopping criteria. The second method is obtained using an embedding strategy,
and again a Kaczmarz--type approach.
We prove well-posedness, stability and convergence of both methods.
\end{abstract}

\section{Introduction}
\label{sec:intro}

In this article  we investigate regularization methods for solving
\textit{linear} and \textit{nonlinear} \textit{systems of ill--posed operator equations}.
Many practical \textit{inverse problems} are naturally formulated in such a way
\cite{boc02, BurHofPalHalSch05, Ep03, HalFid06, KruKisReiKruMil03, LMZ06, LMZ06a, Nat97, Nat01, XuMWan06}.

We consider the problem of determining some physical quantity  $x$ from data $(y^i)_{i=0}^{N-1}$, which is
functionally related by
\begin{equation}\label{eq:f-ix}
    \F_{i}( x )
     =
    y^i  \,,
   \quad
    i = 0, \dots, N-1
    \,.
\end{equation}
Here $\F_i: D_i\tm X \to Y$ are operators between separable
Hilbert spaces $X$ and $Y$.
We are specially interested in the situation where the data is not exactly
known, i.e.,  we have only an approximation  $y^{\delta,i}$ of the exact
data, satisfying
\begin{equation}\label{eq:noisy-i}
    \norm{ y^{\delta,i} - y^i}
    < \delta^i
    \,.
\end{equation}

Standard  methods for the solution of such systems are based on rewriting
(\ref{eq:f-ix}) as a single equation
\begin{equation}\label{eq:f-x}
    \F( x )
     =
    \f y \,,
   \quad
    i = 0, \dots, N-1
    \,,
\end{equation}
where  $\F : = 1/\sqrt{N} \cdot ( \F_0, \dots, F_{N-1} )$ and $\f y = 1/\sqrt{N} \cdot (y^0, \dots,  y^{N-1})$.
There are at least two basic concepts for solving ill posed equations of  the form
(\ref{eq:f-x}):   \textit{Iterative} regularization methods
(cf., e.g., \cite{Lan51, HanNeuSch95, EngHanNeu96, BakKok04, KalNeuSch06})
and  \textit{Tikhonov type} regularization methods
\cite{Tik63b, TikArs77, SeiVog89, Mor93, EngHanNeu96}.
However these methods become inefficient if $N$ is large or the evaluations of $\F_i(x)$ and $\F_i'(x)^\ast$
are expensive.  In such a situation  Kaczmarz--type methods \cite{Kac37, Nat01} which cyclically
consider each equation in (\ref{eq:f-ix}) separately, are much faster \cite{Nat97} and are often the method of
choice in practice. On the other hand, only few theoretical results about regularizing properties of
Kaczmarz methods are available, so far.

The \textit{Landweber--Kaczmarz approach} for the solution of (\ref{eq:f-ix}), (\ref{eq:noisy-i}) analyzed in this
article consists in incorporating a bang-bang relaxation parameter in the classical Landweber--Kaczmarz
method \cite{KowSch02},
combined with a new stopping rule. Namely,
\begin{equation} \label{eq:lwk-lop}
    x_{n+1} = x_{n} - \om_n \F_{[n]}'(x_{n})^*(\F_{[n]}(x_{n})
    - y^{\delta,[n]}) \,,
\end{equation}
with
\begin{equation} \label{eq:skip}
\om_n   := \om_n( \delta,  y^\delta)   =
\begin{cases}
      1  & \norm{\F_{[n]}(x_{n}) - y^{\delta,[n]}}
      > \tau \delta^{[n]} \\
      0  & \text{otherwise}
\end{cases} \,,
\end{equation}
where $\tau > 2$  is an appropriate chosen positive constant
and $[n] := n \mod N  \in \set{0, \dots, N-1}$.
The iteration terminates if all $\omega_n$ become zero within a cycle,
that is if $\norm{ \F_{i}(x_{n}) - y^{\delta,i}} \leq \tau \delta^i$
for all $i \in \set{0, \dots, N-1}$. We shall refer to this method as
\textit{loping Landweber--Kaczmarz method} (\textsc{lLK}).
Its worth mentioning that, for noise free data, $\om_n = 1$ for all $n$ and therefore, in this special situation,
our iteration is identical to the classical Landweber--Kaczmarz method
\begin{equation} \label{eq:lwk-class}
    x_{n+1} = x_{n} - \F_{[n]}'(x_{n})^*( \F_{[n]}(x_{n}) - y^{\delta,[n]} ) \,,
\end{equation}
which is a special case of \cite[Eq. (5.1)]{Nat97}.

However, for noisy data, the \textsc{lLK} method is
fundamentally different to (\ref{eq:lwk-class}): The parameter $\omega_n$
effects that the iterates defined in (\ref{eq:lwk-lop}) become stationary
and \textit{all components} of the residual vector
$\norm{F_i(x_n)- y^{\delta, i}}$ fall below some threshold,
making (\ref{eq:lwk-lop})  a convergent regularization method.
The convergence of the residuals in the maximum norm better exploits the
error estimates (\ref{eq:noisy-i}) than standard methods, where only
\textit{squared  average} $ 1 / N \cdot \sum_{i=0}^{N-1} \norm{F_i(x_n)- y^{\delta, i}}^2$
of the residuals falls below a certain threshold.
Moreover, especially after a large number of iterations, $\omega_n$ will vanish for some $n$.
Therefore, the computational expensive
evaluation of $\F_{[n]}[x_{n}]^*$ might be loped, making the
Landweber--Kaczmarz method in (\ref{eq:lwk-lop}) a fast alternative to conventional
regularization techniques for system of equations.

The second regularization strategy considered in this  article is an embedding approach,
which consists in rewriting (\ref{eq:f-ix}) into an system of equations on the space $X^N$
\begin{equation}\label{eq:v-ixi}
  \F_i( x^i ) = y^i \,,  \quad \;  i = 0, \dots, N-1\,,
\end{equation}
with the additional constraint
\begin{equation}
\label{eq:x-const}
  \sum_{i=0}^{N-1} \norm{x^{i+1} - x^{i}}^2  = 0\,,
\end{equation}
where we set $x^{N} := x^{0}$.
Notice that if $x$ is a solution of (\ref{eq:f-ix}), then the \textit{constant vector}
$(x^i=x)_{i=0}^{N-1}$ is a solution of system (\ref{eq:v-ixi}), (\ref{eq:x-const}),
and vice versa.
This system of equations is solved using a block Kaczmarz strategy
of the form
\begin{eqnarray} \label{eq:lk-block}
  \f x_{n+1/2} &=&
  \f x_{n} - \omega_{n} \f F'(\f x_{n})^\ast ( \f F(\f x_{n})   - \f y^\delta )
  \\ \label{eq:lk-block2}
  \f x_{n+1} &=& \f x_{n+1/2} - \omega_{n+1/2} \f G ( \f x_{n+1/2} ) \,,
\end{eqnarray}
where $\f x := (x^i)_i \in X^N$, $\f y^\delta := (y^{\delta,i})_i \in Y^N$, $\f F (\f x)  :=  (\F_i(x^i))_i \in Y^N$,
\begin{equation} \label{eq:skip2}
\begin{aligned}
    \omega_n   &=
    \begin{cases}
        1  & \norm{ \f F(\f x_n) - \f y^{\delta}}
         > \tau \delta \\
        0  & \text{otherwise}
    \end{cases}\,,
        \\
    \omega_{n+1/2}
    & =
      \begin{cases}
    1  & \norm{ \f G ( \f x_{n+1/2} ) }
      > \tau  \epsilon(\delta)\\
      0  & \text{otherwise}
      \end{cases}\,,
\end{aligned}\end{equation}
with $\delta := \max\set{\delta^i}$. The strictly increasing  function $\epsilon: [0, \infty) \to [0, \infty)$
satisfies $\epsilon(\delta) \to 0$, as $\delta \to 0$, and guaranties the existence of a  finite stopping index.
A natural choice is $\epsilon(\delta) = \delta$.
Moreover, up to a positive multiplicative constant, $\f G$ corresponds to the
steepest descent direction of the functional
\begin{equation} \label{eq:lk-block3}
  \mathcal{G} (\f x) := \sum_{i=0}^{N-1} \norm{x^{i+1} - x^{i}}^2
\end{equation}
on $X^N$. Notice that (\ref{eq:lk-block2}) can also be
interpreted as a Landweber--Kaczmarz step with respect to the equation
\begin{equation}\label{eq:op-D}
  \lambda D( \f x) = 0 \,,
\end{equation}
where $D(\f x) = (x^{i+1}- x^{i})_i  \in  X^N$ and $\lambda$
is a small positive parameter such that $\norm{\lambda D}\leq 1$.
Since  equation (\ref{eq:f-ix}) is embedded into a system of equations on a higher dimensional
function space we call the resulting regularization technique
\emph{embedded Landweber--Kaczmarz} (\textsc{eLK}) method.  As shown in Section \ref{sec:elk}, (\ref{eq:lk-block}), (\ref{eq:lk-block2})  generalizes the Landweber
method for solving (\ref{eq:f-x}).

The article is  outlined as  follows.    In Section \ref{sec:lk} we
investigate the \textsc{lLK} method with the  novel
parameter stopping rule. We prove well--posedness, stability and
convergence, as the noise level tends to zero.
Moreover, we show that all components of the residual  vector fall below
a certain threshold. In Section \ref{sec:elk} we
analyze the \textsc{eLK}  method. In particular,
we make use of the results in Section \ref{sec:lk}  to prove that
the \textsc{eLK}  method is well posed,
convergent and stable.

\section{Analysis of the loping Landweber--Kaczmarz method}
\label{sec:lk}

In this section we present the convergence analysis of the
\textit{loping Landweber--Kaczmarz} (\textsc{lLK}) method.
The novelty of our approach consists in omitting an update in the Landweber Kaczmarz iteration, within one cycle,
if the corresponding $i$--th residual is below some threshold, see (\ref{eq:skip}).
Consequently, the \textsc{lLK} method is not stopped until all residuals
are below the specified threshold.
Therefore, it is the natural counterpart of the Landweber--Kaczmarz iteration
\cite{Kac37, Nat01} for ill--posed problems.

The following assumptions are standard in the convergence
analysis of  iterative regularization methods \cite{EngHanNeu96,HanNeuSch95,KalNeuSch06}.
We assume
that $\F_i$ is \textit{Fr\'echet differentiable} and that there  exists
$ \rho > 0 $ with
\begin{equation} \label{eq:scal}
  \norm{\F_i'(x)}_{Y} \le 1 \,,  \qquad  x  \in B_{\rho}(x_0) \subset
  \bigcap_{i=0}^{N-1} D_{i} \,.
\end{equation}
Here $B_{\rho}(x_0)$ denotes the closed ball of radius
$\rho$ around the starting value $x_0$,  $D_{i}$ is   the domain of $\F_i$,
and $\F_i'(x)$ is the  Fr\'echet derivative of $\F_i$ at $x$.

Moreover, we assume that the \textit{local tangential  cone condition}
\begin{equation} \label{eq:nlc}
\begin{aligned}
  \norm{\F_i(x) - \F_i(\bar{x}) - \F_i'(x)( x - \bar{x})}_Y
  \leq \eta \norm{ \F_i(x)-\F_i(\bar{x}) }_{Y} \,, \\
  x, \bar{x} \in  B_{\rho}  (x_0)  \subset D_{i}
\end{aligned}
\end{equation}
holds for some $\eta < 1 / 2$. This is a central assumption in the analysis  of iterative methods
for the solution of nonlinear ill--posed problems \cite{EngHanNeu96, KalNeuSch06}.

In the analysis of the \textsc{lLK} method we assume that
$\tau$ (used in the definition (\ref{eq:skip}) of $\omega_n$) satisfies
\begin{equation}\label{eq:defitau}
  \tau > 2 \frac{1+\eta}{1-2\eta} > 2 \,.
\end{equation}
Note that, for noise free data, the \textsc{lLK} method
is equivalent to the classical Landweber--Kaczmarz method, since
$\omega_n = 1$ for all $n\in \N$.

In the case of noisy  data, iterative regularization methods require
early termination, which is enforced by an appropriate stopping
criteria. In order to motivate the stopping criteria, we derive in the following
lemma an estimate related to the monotonicity of the sequence $x_n$ defined in (\ref{eq:lwk-lop}).

%In the next Lemma we discuss the new stopping rule proposed
%in this article by the Landweber--Kaczmarz method.

\begin{lemma}\label{lem:mon1}
Let $x$ be a solution of (\ref{eq:f-ix}) where $\F_i$ are  Fr{\'e}chet differentiable
in $B_\rho(x_0)$, satisfying (\ref{eq:scal}), (\ref{eq:nlc}).
Moreover, let  $x_n$ be the sequence defined in (\ref{eq:lwk-lop}), (\ref{eq:skip}).
Then
\begin{equation} \label{eq:mon}
    \begin{aligned}
        &\norm{ x_{n+1} - x}^2 - \norm{x_{ n } - x}^2 \\
        &\leq \om_n \norm{\F_{[n]}(x_{n}) - y^{\delta, [n]}}
        \biggl( 2(1+\eta) \delta^{[n]} - (1-2\eta)\norm { \F_{[n]} (x_{ n } ) - y^{\delta, [n]} } \biggr),
\end{aligned}
\end{equation}
where $[n] = \mod(n, N)$.
\end{lemma}

\begin{proof}
The proof follows the lines of \cite[Proposition 2.2]{HanNeuSch95}.
Notice that if $\omega_n$ is different from zero, inequality (\ref{eq:mon})
follows analogously as in \cite{HanNeuSch95}. In the case $\omega_n=0$, (\ref{eq:mon}) follows from
$x_n = x_{n+1}$.
\end{proof}

Motivated, by Lemma \ref{lem:mon1} we define the termination index
$n_\ast^\delta = n_\ast^\delta(y^\delta)$ as the smallest integer multiple of $N$ such that
\begin{equation} \label{eq:lk-stop}
  x_{n_\ast^\delta} = x_{n_\ast^\delta+1} = \cdots = x_{n_\ast^\delta + N} \,.
\end{equation}
%Here we adopt the notation $\delta  := (\delta^0,\delta^1,\ldots,\delta^{N-1})$.

Now we have the following monotonicity result:

\begin{lemma}\label{lem:mon2}
Let $x$, $\F_i$ and $x_n$ be defined as in Lemma \ref{lem:mon1} and $n_\ast^\delta$ be defined
by (\ref{eq:lk-stop}). Then we have
\begin{equation} \label{eq:mon1}
    \norm{ x_{n+1} - x}
    \leq \norm{x_{n} - x} \;,\quad  n =0,\dots, n_\ast^\delta\;.
\end{equation}
Moreover, the stoping rule (\ref{eq:lk-stop}) implies $\omega_{n_\ast^\delta+i} = 0$ for all
$i \in \set{0, \dots, N-1}$, i.e.,
\begin{equation} \label{eq:res-small}
  \norm{\F_i(x_{n_\ast^\delta}) - y^{\delta,i} } \leq \tau \delta^i \,,
  \quad
  i=0,\dots, N-1 \,.
\end{equation}

\end{lemma}

\begin{proof}
If $\omega_n=0$, then (\ref{eq:mon1}) holds since the iteration stagnates.
Otherwise, from the definitions of $\omega_n$ in (\ref{eq:skip}) and
$\tau$ in (\ref{eq:defitau}), it follows that
\begin{equation}\label{eq:helpA}
      2(1+\eta) \delta^i - (1-2\eta) \norm{ \F_{[n]}(x_n) - y^{\delta,[n]}} <  0 \;,
\end{equation}
and the right hand side in (\ref{eq:mon}) becomes non--positive.

To prove the second assertion we use (\ref{eq:mon}) for $n = n_{\ast}^\delta+i$, for $i \in \set{0, \dots, N-1}$.
By noting that $x_{n_{\ast}^\delta+i} = x_{n_{\ast}^\delta}$ and $[n_{\ast}^\delta+i] = i$, we
obtain
\[
    0 \leq \om_{n_\ast^\delta+i} \cdot \norm{ \F_{i}(x_{n_\ast^\delta} ) - y^{\delta,i} }
    \biggl( 2(1+\eta) \delta^i - (1-2\eta)\norm { y^{\delta,i} - \F_{i}(x_{n_\ast^\delta}) } \biggr) \,,
\]
for $i \in \set{0, \dots, N-1}$. Suppose $\om_{n_\ast^\delta+i} \not= 0$, then
$2 (1+\eta) \delta^i - (1-2\eta) \norm { y^{\delta,i} - \F_{i}( x_{n_\ast^\delta}) } \geq 0$,
which contradicts the definition of $\om_{n_\ast^\delta+i}$.
\end{proof}

Note that for $n > n_\ast^\delta$, $\omega_n \equiv 0$ and therefore
$x_n = x_{n_\ast^\delta}$. This shows that the Landweber--Kaczmarz method becomes
stationary after $n_\ast^\delta$.

\begin{remark}\label{rem:sum}
Similar to the nonlinear Landweber iteration one obtains the
estimate
\begin{equation}\label{eq:finite:stop}
  \frac{n_\ast^\delta \cdot \left( \tau \min_i(\delta^i) \right)^2}{N}   \leq
    \sum_{n=0}^{ n_\ast^\delta - 1}
    \om_n \norm{ y^{\delta,[n]} - \F_{[n]}(x_{n} ) } ^2
    \leq \frac{ \tau \norm{x - x_{n_\ast^\delta}}^2 }{ (1-2\eta)\tau - 2(1+\eta ) }
    \,.
\end{equation}
Here we use the notation of Lemma \ref{lem:mon1}.
\end{remark}

>From Remark \ref{rem:sum} it follows that, in the case of noisy  data,
$n_\ast^\delta < \infty$ and  the iteration terminates after a finite
number of steps.
Next, we state the main result of this section, namely that
the  Landweber--Kaczmarz method   is a convergent
regularization method.

\begin{theorem} \label{convergence}
Assume that  $\F_i$ are Fr\'echet-differentiable in
$B_{\rho}(x_0)$, satisfy (\ref{eq:scal}), (\ref{eq:nlc})
and  the system  (\ref{eq:f-ix})  has a solution in
$B_{\rho/2}(x_0)$. Then

\begin{enumerate}
\item
For exact data $y^{\delta,i} = y^i$, the sequence  $x_n$ in (\ref{eq:lwk-lop}) converges to a
solution of (\ref{eq:f-ix}).
Moreover, if $x^\dagger$ denotes the unique solution of
(\ref{eq:f-ix}) with minimal distance to $x_0$ and
\begin{equation} \label{eq:kern}
      \nur(\F_i'(x^\dagger)) \subseteq \nur(\F_i'(x))
      \;, \quad x \in B_\rho (x_0) \,, \; i \in\set{ 0, \dots, N-1}\,,
\end{equation}
then $x_{n} \to x^\dagger$.

\item
For noisy data the loping
Landweber--Kaczmarz iterates $x_{n_\ast^\delta}$ converge to a solution of
(\ref{eq:f-ix})  as $\delta \to 0$. If in
addition  (\ref{eq:kern}) holds, then $x_{n_\ast^\delta}$ converges to $x^\dagger$
as $\delta \to 0$.
\end{enumerate}
\end{theorem}

\begin{proof}
The proof of the first item is analogous to the proof in \cite[Proposition
4.3]{KowSch02} (see also \cite{KalNeuSch06}).
We emphasize that, for exact data, the iteration (\ref{eq:lwk-lop})  reduces to the
classical Landweber--Kaczmarz method, which allows to apply the corresponding
result of \cite{KowSch02}.

The proof of the second item is analogous to the proof of the
corresponding result for the Landweber iteration as in
\cite[Theorem 2.9]{HanNeuSch95}. For the first case within
this proof, (\ref{eq:res-small}) is required. For the second case
we need the monotony result from Lemma \ref{lem:mon2}.
\end{proof}

In the case of noisy data (i.e. the second item of
Theorem~\ref{convergence}), it has been shown in \cite{KowSch02}
that the Landweber--Kaczmarz iteration
\begin{equation} \label{eq:lwk}
    x_{n+1} = x_{n} - \F_{[n]}'(x_{n})^\ast
    ( \F_{[n]}(x_{n}) - y^{\delta,[n]} ) \,, \quad \\
\end{equation}
is convergent if it is terminated after the $\tilde{n}^\delta$--th step, where $\tilde{n}^\delta$ is the smallest iteration index that satisfies
\begin{equation} \label{eq:disi}
    \norm{\F_{[\tilde n^\delta]}(x_{\tilde{n}^\delta}) -
    y^{\delta,[\tilde n^\delta]}} \leq \tau \delta^{[\tilde n^\delta]} \,.
\end{equation}
Therefore, in general, only one of the components of the residual vector
$\bigl( \norm{\F_i(x_{\tilde{n}^\delta}) - y^{\delta,i} } \bigr)_{i}$
is smaller than  $\tau \delta^i$, namely the \textit{active component}
$\norm{\F_{[\tilde{n}^\delta]}(x_{\tilde{n}^\delta}) - y^{\delta,[\tilde{n}^\delta]} }$.
However, the argumentation in \cite{KowSch02}  is incomplete,
in the sense that the case when $\tilde{n}^\delta$ stagnates as $\delta \to 0$,
has not been considered. Hence,  \cite[Theorem 4.4]{KowSch02} requires the additional
assumption that $\tilde{n}^\delta \to \infty$, as $\delta \to 0$, which is usually the  case in practice.

\section{Analysis of the embedded Landweber--Kaczmarz method}
\label{sec:elk}

In the \textit{embedded Landweber--Kaczmarz} (\textsc{eLK}) method for the solution of (\ref{eq:f-ix}),
$x \in X$ is substituted by a vector $\f x = (x^i)_{i=0}^{N-1}$.
In (\ref{eq:lk-block}) each component of $\f x$ is updated independently according to one of the system equations.
In the balancing step (\ref{eq:lk-block2}), the difference between the components of $\f x$  is minimized.

In order to determine $\f x_{n+1/2}$, each of its components
$x^i_{n+1/2}$ can be evaluated independently:
\[
    x^i_{n+1/2} = x^i_{n} - \omega_n  \F_{i}'(x^i_n)^\ast
\left( \F_{i}(x^i_{n}) - y^{\delta,i} \right) , \
i = 0, \dots, N-1 \, .
\]
In the balancing step (\ref{eq:lk-block2}), $\f x_{n+1}$ is determined from
$\f x_{n+1/2}$ by a matrix multiplication with the sparse
matrix $I_{X^N} - \omega_{n+1/2} \f G$, where
\[
    \f G = \lambda^2
    \left(
        \begin{array}{ccccc}
            \nm 2I &   - I  &  \nm 0   &        &   - I \\
            - I  & \nm 2I &  \ddots  & \ddots &       \\
            \nm 0  & \ddots &  \ddots  & \ddots & \nm 0 \\
                & \ddots &  \ddots  & \nm 2I &   - I \\
            - I    &        &  \nm 0   &   - I  & \nm 2I
        \end{array}
    \right)
    \in \mathcal{L}(X^N, X^N) \,.
\]
Here $\lambda$ is a small positive parameter such that $\norm{\lambda D}\leq 1$,
and the operator $\f G$ is a discrete variant of $-\la^2$ times the second derivative operator
and therefore penalizes for varying components.
As already mentioned in the introduction, the the balancing step (\ref{eq:lk-block2}) is a
Landweber--Kaczmarz step with respect to the equation (\ref{eq:op-D}).
The operator $D$ is linear and bounded,  which guaranties the existence of a positive
constant $\lambda$ such that $\lambda D$  satisfies (\ref{eq:scal}), which will be
needed in the analysis of the embedded Landweber method.
The iteration defined in (\ref{eq:lk-block}), (\ref{eq:lk-block2}) is terminated when for the first time
\begin{equation}\label{eq:lke:stop}
    \f x_{{n_\star^\delta}+1} = \f x_{{n_\star^\delta} + 1/2}= \f x_{{n_\star^\delta}}\;.
\end{equation}
The artificial  noise level  $\epsilon: [0, \infty) \to [0, \infty)$ satisfies
$\epsilon(\delta) \to 0$, as $\delta \to 0$ and guaranties the existence of
a  finite stopping index  in  the  \textsc{eLK} method.

In the sequel we shall apply the results of the Section \ref{sec:lk} to prove convergence
of the \textsc{eLK} method.
As initial guess we use a constant vector $\f x_{0} := (x_{0})_i$
whose components are identical to $x_{0}$. Moreover, our convergence analysis will again require
the scaling assumption (\ref{eq:scal}) and the tangential cone condition (\ref{eq:nlc}) to be
satisfied near $x_{0}$.

% \begin{psfrags}
% \psfrag{L1}{$l_0$}
% \psfrag{L2}{$l_1$}
% \psfrag{L3}{$l_2$}
% \psfrag{x0}{$x_{n}$}
% \psfrag{x1}{$x_{n+1/2}^0$}
% \psfrag{x2}{$x_{n+1/2}^1$}
% \psfrag{x3}{$x_{n+1/2}^2$}
%
\begin{figure}
\begin{center}
    \includegraphics[width=0.4\textwidth]{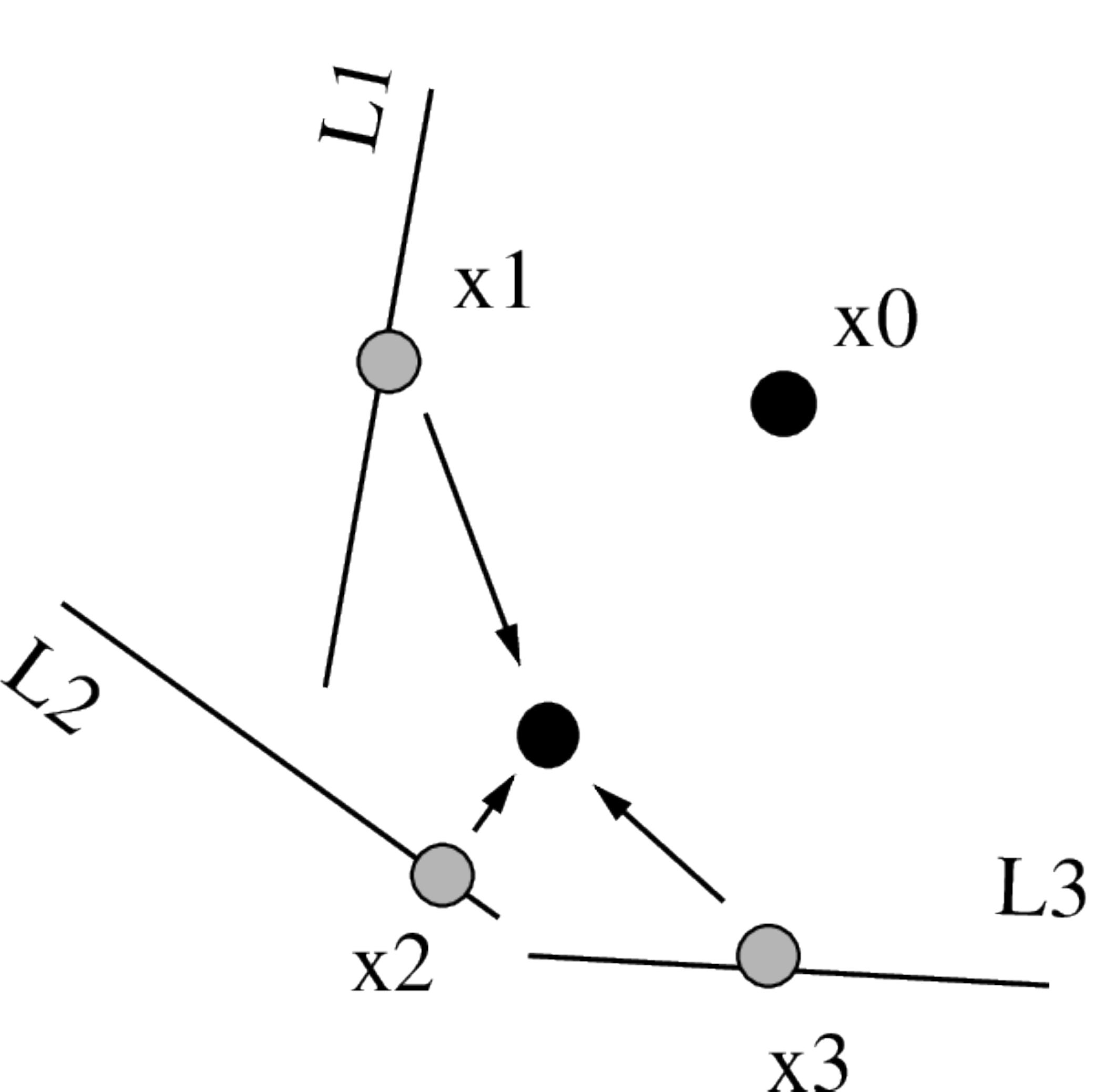}\hspace{1cm}
    \includegraphics[width=0.4\textwidth]{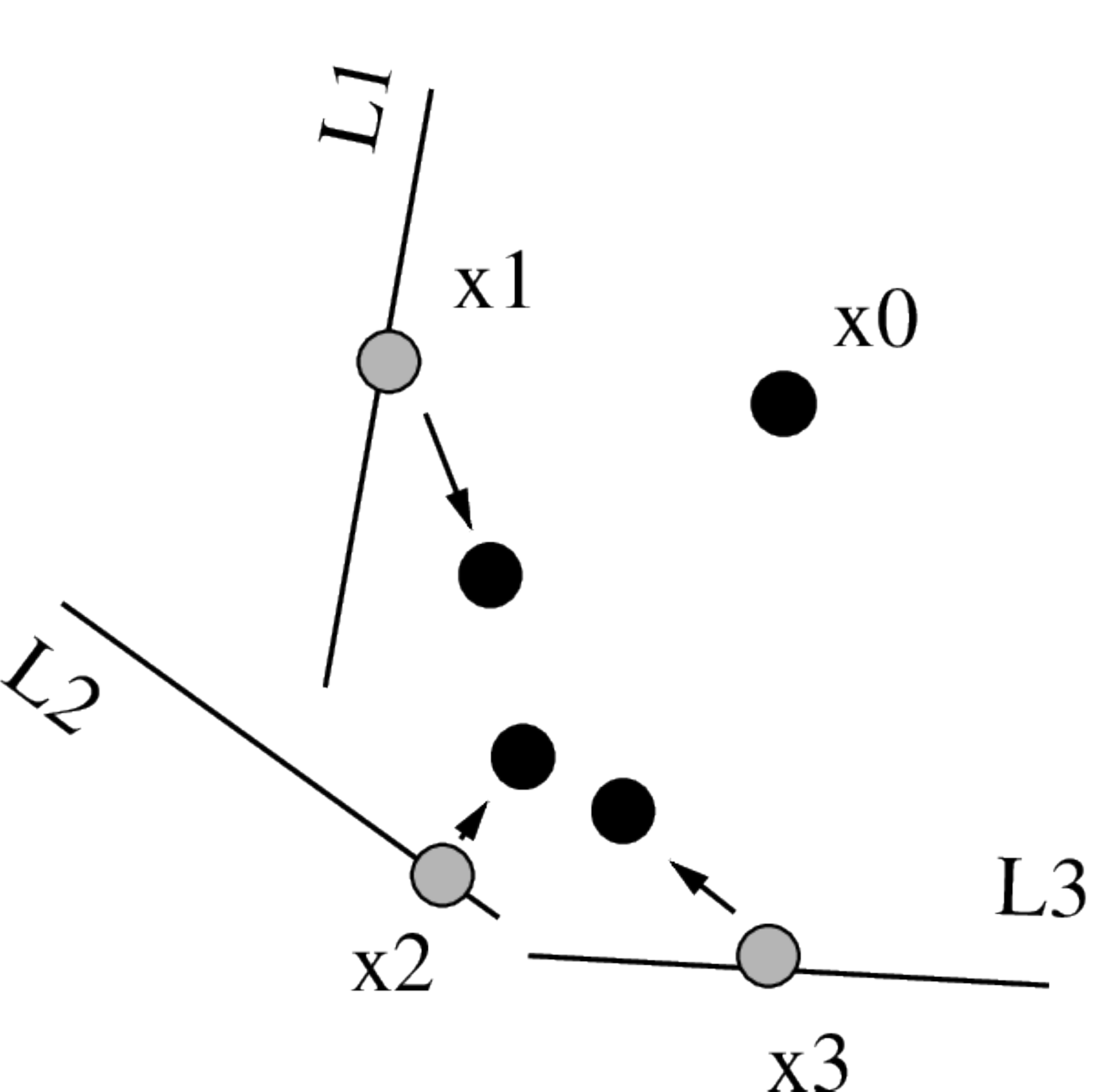}
\end{center}
\caption{Landweber versus \textsc{eLK} method for solving
$a_i \cdot x = y^i$, $i \in \set{0,1,2}$ with $\norm{a_1}=1$. In this case
$x_{n+1/2}^i := x_{n} - a_i^* ( a_i \cdot x_n - y^i )$ is the orthogonal
projection of $x_{n}$ on $l_i := \set{x\in \R^2: a_i \cdot x = y^i}$.
Each step of the Landweber iteration (left) generates a single element in $X$,
namely the average $1/3 \sum_{i=0}^2 x_{n+1/2}^i$. In contrast, a cycle in the embedded
Landweber--Kaczmarz method (right) generates a vector in $X^N$, where each component of
$\f x_{n+1}$ is a linear combination of $x_{n+1/2}^i$.}
\label{fig:emb}
\end{figure}
%
% \end{psfrags}

\begin{remark}[Comparison with the  classical Landweber iteration]
Let $\F := 1/\sqrt{N} \cdot ( \F_0, \dots, F_{N-1})$ and
$\f y^\delta  :=  1/\sqrt{N}  \cdot (y^{\delta,0}, \dots , y^{\delta,N-1})$.
The Landweber iteration for the solution of $\F(x) = \f y^\delta$, see (\ref{eq:f-x}), is
\cite{HanNeuSch95, EngHanNeu96}
\begin{equation*}
    \begin{aligned}
    x_{n+1}
    & =
    x_{n} -
    \F'[x_{n}]^\ast (\F  (x_{n}) - \f y^\delta )
    \cr
    & =
    x_{n} -
    \frac{1}{N} \cdot
    \sum_{i=0}^{N-1}
    \F_i'(x_{n})^\ast (\F_i (x_{n}) - y^{\delta,i})
    \cr & =
    \frac{1}{N} \cdot
    \sum_{i=0}^{N-1}
    \biggl( x_{n} - \F_i'(x_{n})^\ast (\F_i (x_{n}) - y^{\delta,i}) \biggr)
    \,.
    \end{aligned}
\end{equation*}
If we set $x_{n+1/2}^i := x_n - \F_i'(x_{n})^\ast (\F_i (x_{n}) - y^{\delta,i})$
then the Landweber method can be  rewritten in form similar to
(\ref{eq:lk-block}), (\ref{eq:lk-block2}), namely
\begin{eqnarray} \nonumber
   x_{n+1/2}^i
   &=&
   x_{n} -
   F'_i(x_{n})^\ast ( F_i(x_{n} )   - y^{\delta,i} ) \,,
  \\ \label{eq:lw-class2}
  x_{n+1} & =&
  \frac{1}{N} \cdot
    \sum_{i=0}^{N-1}
    x_{n+1/2}^i  \,.
\end{eqnarray}
Hence, the distinction between the Landweber and the \textsc{eLK} method is that (\ref{eq:lw-class2}) in the Landweber method makes all
components equal, whereas the balancing step (\ref{eq:lk-block2}) in the embedded
Landweber--Kaczmarz method leaves them distinct.

In order to illustrate the idea behind the \textsc{eLK} method, we exemplarily consider the case $N = 3$. In this case the mechanism of
the embedded iteration is explained in Figure \ref{fig:emb} in contrast to the
Landweber method.
\end{remark}

In the next theorem we prove that the termination index is well defined, as well as
convergence and stability of the \textsc{eLK} method.

\begin{theorem} \label{thm:convergence}
Assume that  the operators $\F_i$ are Fr\'echet-differentiable in
$B_{\rho}(x_0)$ and satisfy (\ref{eq:scal}), (\ref{eq:nlc}).
Moreover, we assume that (\ref{eq:f-ix})  has a solution in $B_{\rho/2}(x_0)$.
Then we have:
\begin{enumerate}
\item For exact data $y^{\delta,i} = y^i$, the sequence  $\f x^n$ in
      (\ref{eq:lk-block}), (\ref{eq:lk-block2}) converges to a
      constant vector $(x)_i$, where $x$ is a solution (\ref{eq:f-ix}) in
      $B_{\rho/2}(x_0)$.
      Additionally, if the operators $\F_i$ satisfy (\ref{eq:kern}), then
      the sequence $\f x^n$ converges to the constant vector
      $\f x^\dag = (x^\dag)_i$,
      where $x^\dag$ is the unique solution of minimal distance
      to $x_{0}$.

\item For noisy data $\delta >0$, (\ref{eq:lke:stop}) defines a finite termination index
      $n_\star^\delta$.
      Moreover, the embedded Landweber--Kaczmarz iteration $\f x^{n_\star^\delta}$
      converges to a constant vector $\f x = (x)_i$, where $x$ is a solution (\ref{eq:f-ix}) in
      $B_{\rho/2}(x_0)$, as $\delta \to 0$.
      If in addition (\ref{eq:kern}) holds, then each component of
      $\f x^{n_\star^\delta}$ converges to $x^\dag$, as $\delta \to 0$.
\end{enumerate}

\end{theorem}

\begin{proof}
In order to prove the first item we apply Theorem \ref{convergence}, item 1 to the system
(\ref{eq:v-ixi}), (\ref{eq:op-D}). From (\ref{eq:scal}) it follows that $\norm{\f F[\f x]} \leq 1$ for  $ \f x \in B_{\rho}(x_0)^N$.
Moreover, since $D$ is bounded linear, $\norm{\lambda D} \leq 1$ for sufficiently small
$\lambda$. The tangential cone condition (\ref{eq:nlc}) for $\F_i$ implies
\[
\begin{aligned}
  \norm{\f F(\f x) - \f F(\bar{\f x}) - \f F'(\f x)( \f x - \bar{\f x})}
  \leq \eta \norm{ \f F(\f x)-\f F(\bar{\f x})} \,, \cr
   \f x, \bar{\f x} \in  B_{\rho}(x_0)^N  \,.
 \end{aligned}
\]
Moreover, since $\lambda D$ is a linear operator, the tangential cone condition
is obviously satisfied for $\lambda  D$ with the same $\eta$.
Therefore, by applying Theorem \ref{convergence}, item 1
we conclude that $\f x_{n}$ converges to a solution $\tilde {\f x}$
of (\ref{eq:f-ix}), (\ref{eq:op-D}). From (\ref{eq:op-D}) it follows that $\tilde {\f x}= (\tilde{x})_i$ is a
constant vector. Therefore, $\F_i(\tilde x) = y^i$, proving the assertion.

Additionally, let  $\f x^\dag$ denote the solution of (\ref{eq:v-ixi}), (\ref{eq:op-D})  with minimal distance to
$(x_{0})_i$. As an auxiliary result we show that $\f x^\dag = (x^\dag)_i$, where $x^\dag$ is the
unique solution of (\ref{eq:f-ix}) with minimal distance to $x_0$.
Due to (\ref{eq:op-D}) we have $\f x^\dag = (\tilde x)_i$,
for some $\tilde x \in X$. Moreover, the vector $(x^\dag)_i$ is a solution of (\ref{eq:f-ix}), (\ref{eq:op-D})
and
\[
    \norm{\tilde x - x_{0}}^2 = \frac{1}{N} \sum_{i=0}^{N-1}\norm{\tilde x - x_{0}}^2
\leq
    \frac{1}{N} \sum_{i=0}^{N-1}\norm{x^\dag - x_{0}}^2 = \norm{x^\dag - x_{0}}^2 \,.
\]
Therefore $\f x^\dag = (x^\dag)_i$. Now, if (\ref{eq:kern}) is satisfied, then
\[
    \mathcal{N}( \f F'(\f x^\dagger) ) \subseteq \mathcal{N}(\f F'(\f x)) \;, \quad
    \f x \in B_\rho(x_0)^N
\]
and by applying Theorem \ref{convergence} we conclude that $\f x^n \to \f x^\dag$.

The proof of the second item follows from Theorem \ref{convergence}, item  2
in an analogous way as above.
\end{proof}

As consequence of Theorem \ref{thm:convergence}, if $n_\star^\delta$ is
defined by (\ref{eq:lke:stop}),
$\f x^{n_\star^\delta} = (x^{n_\ast^\delta}_i)_i$, then
\begin{equation}\label{eq:average}
    x^{n_\star^\delta} := \sum_{i=0}^{N-1} x^{n_\star^\delta}_i  \longrightarrow x^\dag
    \,,
\end{equation}
as $ \delta \to 0 $.
However, Theorem \ref{thm:convergence} guaranties even more: All components
$x^{n_\star^\delta}_i$ converge to $x^\dag$ as the noise level tend to zero.
Moreover, due to the averaging process in (\ref{eq:average}) the noise level in the
actual regularized solution $x^{n_\star^\delta}$ becomes
noticeable   reduced.

\section{Conclusion} \label{sec:conclusion}

We have suggested two novel Kaczmarz  type regularization techniques for solving
systems of ill-posed operator equations. For each one we proved
convergence and stability results. The first technique is a variation of the
Landweber--Kaczmarz method with a new stopping rule and a
loping parameter that allows to skip some of the inner cycle iterations,
if the corresponding residuals are sufficiently small.
The second method derives from an embedding strategy, where the original system is
rewritten in a larger space.

One advantage of Kaczmarz type methods is the fact that the
resulting regularizing  methods better explore the special structure of the model and
the pointwise noise-estimate $\norm{y^i-y^{\delta,i}} < \delta^i$.
Moreover, for noisy data, it is often much faster \cite[p.19]{Nat97}
in practise than Newton--type methods. The key in this article to prove convergence,
as the  noise level tends to zero, was to introduce a bang--bang relaxation
parameter $\omega_n$ in the iteration (\ref{eq:lwk-lop}). Recently regularizing \textit{Newton--Kaczmarz} methods, similar to
(\ref{eq:lwk-class}), have been analyzed  \cite{BurKal06}. Their
convergence analysis was based on the assumption \cite[Eq. (3.14)]{BurKal06} which
implies that, for exact data, a single equation in (\ref{eq:f-ix}) would
already be sufficient to find the solution of (\ref{eq:f-ix}).
Our strategy of incorporating a bang-bang relaxation parameter in (\ref{eq:lwk-lop}), which can also be
combined with Newton type iterations \cite{Han97}, overcomes this severe restriction.
An analysis  of loping Kaczmarz--type Levenberg--Marquard \cite{Han97} and
steepest--descent \cite{Sch96} regularization methods will be presented in a
forthcoming publication.

Our methods allow fast implementation. The effectiveness is presented in
a subsequent article: There we shall consider the Landweber--Kaczmarz methods of
Sections \ref{sec:lk}, \ref{sec:elk} applied to thermoacoustic tomography
\cite{BurHofPalHalSch05, HalFid06,XuMWan06}, semiconductor equations \cite{LMZ06, LMZ06a} and
Schlieren imaging \cite{Bre98, LedZan99}.

\section*{Acknowledgements}

The work of M.H. and O.S. is supported by FWF (Austrian Fonds zur
F\"orderung der wissenschaftlichen Forschung) grants  Y--123INF and P18172--N02.
Moreover, O.S. is supported by FWF projects FSP S9203 and S9207.
The work of A.L. is supported by the Brasilian National Research Council CNPq,
grants 305823/2003--5 and 478099/2004--5. The authors thank Richard Kowar
for stimulating discussion on Kaczmarz methods.

\bibliographystyle{amsplain}
\bibliography{lit-embedding}

\medskip
        {\it E-mail address: }markus.haltmeier@uibk.ac.at\\
 \indent{\it E-mail address: }a.leitao@ufsc.br\\
 \indent{\it E-mail address: }otmar.scherzer@uibk.ac.at\\

\end{document}